\documentclass[12pt]{article}
\usepackage{graphicx}
\usepackage{amssymb}
\usepackage{amsmath}
\usepackage{amsthm}
\usepackage{float}
\usepackage{bm}
\usepackage{color}
\usepackage{mathrsfs}
\usepackage{lipsum}
\usepackage{amsfonts}
\usepackage{graphicx}
\usepackage{epstopdf}
\usepackage{algorithmic}

\newtheorem{thm}{Theorem}{\bf}{\it}
\newtheorem{rmk}{Remark}{\bf}{\it}
\newtheorem{prop}{Proposition}{\bf}{\it}
\newtheorem{lem}{Lemma}{\bf}{\it}

\begin{document}

\title{Global Existence of the Non-isothermal Poisson--Nernst--Planck--Fourier System\thanks{This work was funded by the NSF grant DMS-1759536 (C. Liu), DMS-1729589 (P. Liu).}}

\author{Chia-Yu Hsieh\thanks{Department of Mathematics, The Chinese University of Hong Kong, Shatin, N.T., Hong Kong, email: {\tt cyhsieh@math.cuhk.edu.hk}} \and Tai-Chia Lin\thanks{Department of Mathematics, Center for Advanced Study in Theoretical Sciences (CASTS), National Center for Theoretical Sciences (NCTS), National Taiwan University, Taipei 10617, Taiwan, email: {\tt tclin@math.ntu.edu.tw}} \and Chun Liu\thanks{Department of Applied Mathematics, Illinois Institute of Technology, Chicago, IL 60616, USA, email: {\tt cliu124@iit.edu}} \and Pei Liu\thanks{School of Mathematics, University of Minnesota, Twin Cities, MN 55455, USA, email: {\tt liu01304@umn.edu}}}
\date{}

\allowdisplaybreaks[1]

\maketitle

\begin{abstract}
In this paper, we consider a non-isothermal electrokinetic model, which is derived from the Energetic Variational Approach. The charge transport is described through the Poisson--Nernst--Planck equations with variable temperature, and the heat flux satisfies the Fourier's law. This Poisson--Nernst--Planck--Fourier model satisfies both the first law and second law of thermodynamics as well as the Onsager's reciprocal relations, thus it is thermodynamic-consistent. Finally, we prove the global well-posedness for this model under the smallness assumption of the initial data by the energy method.
\end{abstract}

\noindent \textbf{Keywords:} Non-Isothermal Fluid, Variational Principle, Global Existence \\
\noindent \textbf{Mathematics Subject Classification (2010):} 82C21, 82D15, 35Q92, 49S05

\section{Introduction}
\label{intro}
The classical Poisson--Nernst--Planck (PNP) theory as well as its various modifications \cite{Schuss2001,Gillespie2002,CKC:BJ:03,Li:N:2009,HEL:CMS:11,HTLE:JPC:12,Wei2012,XML:PRE:2014,LiujiXu2016,Wang2016,Chen2017,Song2018} have been widely studied in describing the dynamics of mobile charges in electrolytes, with real applications such as the energy devices \cite{WVP:EA:11,Soestbergen2010,Lai2011,Kondrat2013,Pilon2015,Ike2016} and the biological ion channels \cite{CCK:BJ:00,GKCN:JPC:04,Singer2008,Lu2010,Gardner2011,Liu2014,Ji2015}. The PNP theory assumes the physical system to be isothermal and described using a constant temperature, which might work well for those systems with negligible temperature variations. However, the electric current is usually accompanied with the Joule heating effect, resulting in considerable temperature changes. Moreover, the thermoelectric effects indicate there exists direct conversion between the temperature gradient and the voltage gradient. The coupling between the temperature evolution and the electric currents is very important in understanding the dynamic properties of the electrolyte systems and has attracted extensive interests in recent years \cite{Kafoussias1995,Xu2012a,Das2014,Berrueta2014,Entremont2015,Hooshmand2015,Sarwar2016,Jiang2016,Janssen2017,Ma2018}.

In literature, there are many works on the analytical properties of the classical PNP equations and its modifications, such as the well-poseness \cite{Mock1974a,Ogawa2008,Deng13,Kinderlehrer2017,Hsieh2019a}, long time behavior \cite{Mock1975,Arnold1999,Poincar2000,Hsieh2015a}, singular pertubation \cite{Singer2008,Liu2015e,Ji2018,Aitbayev2018} and the references therein. Many efforts have been also devoted to the non-isothermal models, for instance, \cite{Feireisl2008} derived a priori estimates and the weak stability of the compressible Navier--Stokes--Fourier system; \cite{Bulicek2009} obtained the existence and long-time behavior of the incompressible Navier--Stokes--Fourier system; \cite{Eleuteri2015} proved the existence of a two-phase diffuse interface model of incompressible fluids; \cite{Anna2019} analyzed the global-in-time well-posedness of the non-isothermal liquid crystal system.

In this paper, we present some analytical results on a non-isothermal Poisson--Nernst--Planck--Fourier (PNPF) system, which is derived from the Energetic Variational Approach (EnVarA) \cite{Hyon2010a,SXu2014,LWL:CMS:2018}. Compared with most of the non-isothermal models in history, our model is thermodynamic consistent in the sense of satisfying basic thermodynamic laws as well as the Onsager reciprocal relations \cite{Onsager1931d,Onsager1931b}. The rest of the paper is organized as follows. In section 2, we present the model derivation and discussion. In section 3, we describe the result of global existence of the Poisson--Nernst--Planck--Fourier system. In section 4, we give the proof details of the well-posedness result by the energy method.

\section{Model Derivation}
\label{sec:1}

The Energetic Variational Approach (EnVarA) is a systematic procedure to derive the consistent constitutive relations and thus the dynamic equations. It has three steps: (1) Least Action Principle: with given form of the total energy as a functional of the flow-map, the conservative force minimizes the total action; (2) Maximum Dissipation Principle: with given form of the entropy production, the dissipative force maximizes the energy dissipation; (3) Onsager Principle: the conservative force should be balanced with the dissipative force.

The original EnVarA is for general classical mechanics \cite{Rayleigh,Onsager1931d,Onsager1931b}, and has shown to be successful in general complex fluids in the isothermal case. When the temperature is a variable, it turns out the EnVarA is consistent with the basic thermodynamic laws \cite{LWL:CMS:2018}. Here we derive all the constitutive relations through the EnVarA, demonstrating the variational structure of the Poisson--Nernst--Planck--Fourier system.

Compared with the EnVarA in the isothermal case, whose properties are determined by the fluid flow-map, we also need to take into account the energy flow-map in the non-isothermal case. In another word, there are two type of kinematic relations:

\begin{equation}
\text{Mass conservation:} \begin{cases}
p_t + \nabla \cdot (p v^p) = 0,\\[3mm]
n_t + \nabla \cdot (n v^n) = 0,
\end{cases}\label{ma_con}
\end{equation}
\begin{equation}
\text{Energy conservation:}
\begin{cases}
e^p_t + \nabla \cdot (e^p v^p) =  -\nabla \cdot (P^p v^p) - pv^p \cdot \nabla\phi - \nabla \cdot q^p + q^{pn},\\[3mm]
e^n_t + \nabla \cdot (e^n v^n)=  -\nabla \cdot (P^n v^n) +nv^n \cdot \nabla \phi - \nabla \cdot q^n + q^{np}.
\end{cases}\label{en_con}
\end{equation}
For simplicity, we only consider the electrolytes in whole space $\mathbb{R}^3$ with two ionic species of valences $z=\pm 1$. The positive ion density distribution is denoted as $p$ and the negative ion density is $n$. $v^p$ and $v^n$ are their velocities separately. The energy conservation comes from the First Law of Thermodynamics. $e^p=c^p p \theta$ and $e^n=c^n n \theta$ represents the internal energy density of the positive/negative charges, without the electrostatic potential energy. $\theta$ is the temperature. $c^p$ and $c^n$ are the specific heat capacitance of the ions. It should be noted that, the internal energy transports along with the material, so that $e^p$ transports with the velocity $v^p$ while $e^n$ transports with the velocity $v^n$. $P^p=p\theta$ and $P^n=n\theta$ are the thermo-pressure of ions, $P^p v^p$ and $P^n v^n$ describe the rate of work from pressure. $q^p$ and $q^n$ represents the heat fluxes within the corresponding ionic species, $q^{pn}=-q^{np}$ represents the heat exchange between the two species. $\phi$ is the electrical potential satisfying the free-space Poisson's equation,
\begin{equation*}
\Delta \phi = n-p, \ \ \ \lim_{|x|\to \infty} \phi =0.
\end{equation*}
Here we used a single variable $\theta$ as the temperature distribution for both the anions and cations. It is sufficient to consider the total energy conservation without distinguishing $q^p$, $q^n$ and $q^{np}$. Adding together the two energy conservation equations in \eqref{en_con}, we obtain,
\begin{eqnarray}
e_t  &=&  -\nabla \cdot (c^p p \theta v^p-c^n n \theta v^n)-\nabla \cdot (P^p v^p+P^n v^n)  \label{energy_con1}\\
&&\qquad\qquad -\nabla \cdot (p \phi v_p -n \phi v_n) - \frac{1}{2}\nabla \cdot (\phi_t \nabla \phi-\phi\nabla\phi_t) - \nabla \cdot q \notag\\
&\triangleq& -\nabla \cdot j^e.
\label{energy_con}
\end{eqnarray}
Here $\displaystyle e=(c^p p +c^n n)\theta + \frac{p-n}{2}\phi$ is the total energy density, including the electrostatic energy.  $q=q^p+q^n$ is the total heat flux. The first term on the right hand side of \eqref{energy_con1} is the transport of the internal energy $e^p$ and $e^n$. The second term represents the work from thermo-pressure. The third term can be viewed as the transport of the electrostatic energy. It should be noted that, the transport of the electrostatic energy is different from the transport of $e^p$ and $e^n$. This is due to the long-range Coulomb pairwise interaction, which also introduces the fourth term, understood as the exchange of electrostatic energy between particles at different locations.

To apply the EnVarA, besides the kinematic relations provided by the conservation laws, we also need the energy equality from the Second Law of Thermodynamics,
\begin{equation}
\frac{d}{dt} S = \triangle \geq 0,
\label{energy_law}
\end{equation}
where $S$ is the entropy and $\triangle$ is the rate of the entropy production. Under the PNP framework, the form of the entropy is treated at the mean-field level,
\begin{eqnarray}
S(t) &= &-\int_{\mathbb{R}^3} p(x,t) [\log p(x,t) - c^p \log \theta(x,t) ] \nonumber\\
& &\qquad\qquad + n(x,t) [\log n(x,t) - c^n \log \theta(x,t) ] dx \nonumber\\
&\triangleq &\int [\eta^p(x,t) + \eta^n(x,t)] dx = \int \eta(x,t) dx.\label{entropy}
\end{eqnarray}
Here $\eta$ is the total entropic density, $\eta^p \triangleq  -p (\log p - c^p \log \theta )$ and $\eta^n\triangleq  - n (\log n - c^n \log \theta )$ represent the entropic densities of individual species.

\paragraph{Least Action Principle} The particle flow-maps describe the mapping from the Lagrangian coordinate to the Eulerian coordinate of the two ionic species, which is usually given by, $x^p_t(X,t) = v^p(x^p,t)$ and $x^n_t(X,t) = v^n(x^n,t)$, with $X$ being the Lagrangian coordinates and $x^p$, $x^n$ being the Eulerian coordinates. However, this approach cannot describe the heat flux, since there is no ``heat (temperature) velocity'' analogous to the mechanical velocities. In fact, the internal energy transports along with the flow-map, while at the same time performing work and absorbing heat as described in \eqref{energy_con1}, which means the kinematics of temperature is much more complicated when coupled with material flow-maps.

It should be noted that, these flow-maps can also be characterized through the accumulated fluxes in Eulerian coordinates alone without pulling back to the Lagrangian coordinates:
\begin{equation*}
J^p_t(x,t) = j^p(x,t) \triangleq p v^p(x,t), \quad  J^n_t(x,t) = j^n(x,t) \triangleq nv^n(x,t).
\end{equation*}
Similarly, the energy flow-map is considered through
\begin{equation*}
J^e_t(x,t) = j^e(x,t).
\end{equation*}

The total action is defined based on the left hand side of the energy equality \eqref{energy_law},
\begin{equation*}
\mathcal{A}[J^p(x,t),J^n(x,t),J^e(x,t)] = -\int_0^\tau S(t) dt.
\end{equation*}
It is a functional of two particle flow-maps $J^p$ and $J^n$ as well as the energy flow-map $J^e$.
\begin{rmk}
The action is usually defined as $\mathcal{A} = \displaystyle\int_0^\tau F(t) dt$ in the isothermal cases, because the energy law is $\dfrac{d}{dt} F = -\triangle_F$, with $F$ being the Helmholtz free energy and $\triangle_F$ being the energy dissipation. In the non-isothermal case, the definition of action should be modified accordingly. 
\end{rmk}

The conservative forces can be computed through the variation of the action with respect to the flow-maps (see Appendix \ref{ap_lap} for detailed calculation), which is the Least Action Principle,
\begin{equation}
\begin{cases}
f^p_{con}(x,t) &=\dfrac{\delta \mathcal{A}}{\delta J^p(x,t)} \\[3mm]
&= -\nabla \left[\log p - c^p \log  \theta  +\dfrac{\phi}{2\theta}+ \int_{\mathbb{R}^3} \dfrac{p(y,t)-n(y,t)}{8\pi \theta(y) |x-y|}dy \right],\\[3mm]
f^n_{con}(x,t) &=\dfrac{\delta \mathcal{A}}{\delta J^n(x,t)} \\[3mm]
&= -\nabla \left[\log n - c^n \log  \theta  -\dfrac{\phi}{2\theta}- \int_{\mathbb{R}^3} \dfrac{p(y,t)-n(y,t)}{8\pi \theta(y) |x-y|}dy \right],\\[3mm]
f^e_{con}(x,t) &=\dfrac{\delta \mathcal{A}}{\delta J^e(x,t)} = \nabla \dfrac{1}{ \theta }.
\end{cases}\label{conservative force}
\end{equation}

\paragraph{Maximum Dissipation Principle}

The form of the entropy production rate is chosen to be in the quadratic form, the same as in the classical PNP model,
\begin{equation}
\triangle = \int_{\mathbb{R}^3} \left[\frac{|j^p|^2}{ D^p p \theta}+\frac{|j^n|^2}{ D^n n \theta}+\frac{1}{k}\left|\frac{q}{  \theta}\right|^2 \right]dx \triangleq \int_{\mathbb{R}^3} \widetilde{\triangle} dx.\label{entropy_production}
\end{equation}
Here $D^p$ and $D^n$ are the mobilities of the ions, related with their diffusion coefficients. $k$ is the heat conduction rate. $\displaystyle \widetilde{\triangle}$ is the density of entropy production rate. Then the dissipative forces can be computed from the variation of the entropy production rate with respect to the fluxes (see Appendix \ref{ap_mdp} for detailed calculation), which is the Maximum Dissipation Principle,
\begin{equation}
\begin{cases}
f^p_{dis}(x,t) &=\dfrac{1}{2}\dfrac{\delta \triangle}{\delta j^p(x,t)} \\[3mm]
&= \dfrac{j^p}{D^p p\theta} - \dfrac{(c_p+1) \theta +\phi }{k\theta^2}q \\[3mm]
&\qquad +\nabla \displaystyle\int_{\mathbb{R}^3} \dfrac{q(y,t)}{2k \theta^2(y,t)} \left(\dfrac{\phi(y,t) (y-x)}{4\pi|x-y|^{3}} - \dfrac{\nabla_y \phi(y,t)}{4\pi |x-y|} \right) dy,\\[3mm]
f^n_{dis}(x,t) &=\dfrac{1}{2}\dfrac{\delta \triangle}{\delta j^p(x,t)} \\[3mm]
&= \dfrac{j^n}{D^n n\theta} - \dfrac{(c_n+1) \theta -\phi  }{k\theta^2}q \\[3mm]
&\qquad -\nabla \displaystyle\int_{\mathbb{R}^3} \dfrac{q(y,t)}{2k \theta^2(y,t)} \left(\dfrac{\phi(y,t) (y-x)}{4\pi|x-y|^{3}} - \dfrac{\nabla_y \phi(y,t)}{4\pi |x-y|} \right) dy,\\[3mm]
f^e_{dis}(x,t) &=\dfrac{1}{2}\dfrac{\delta \triangle}{\delta j^e(x,t)} = \dfrac{q}{k\theta^2}.
\end{cases}\label{dissipative force}
\end{equation}
Here $1/2$ corresponds to the quadratic form of $\Delta$, meaning it is a linear response theory.

\paragraph{Onsager's Principle} Using the conservative and dissipative forces, \eqref{energy_law} can be rewritten into the form
\begin{equation}
\int_{\mathbb{R}^3} [(f^p_{cons}-f^p_{dis})\cdot j_p + (f^n_{cons}-f^n_{dis})\cdot j_p ]dx = \int_{\mathbb{R}^3} (f^e_{dis}-f^e_{cons})\cdot j_e dx. \label{balance}
\end{equation}
The left hand side is about the mechanics (deformation) of the fluid system, while the right hand side is about the energy (heat). It would be nature to conclude the balance between the conservative and dissipative forces, which is known as the Onsager's Principle.

The force balance between $f^e_{con}$ and $f^e_{dis}$ gives the Fourier's law,
\begin{equation*}
q = -k\nabla \theta.
\end{equation*}
The force balances on the ion flow-maps reveal the Darcy's law,
\begin{equation*}
\begin{cases}
pv^p=j^p = -D^p [\nabla (p\theta) +p \nabla \phi],\\[3mm]
nv^n=j^n = -D^n [\nabla (n\theta) -n \nabla \phi].
\end{cases}
\end{equation*}
\begin{rmk}
In this derivation, we ignored the contribution from the solute, which can be taken into account by introducing the solute heat capacitance and relative drag between solvent and solute molecules \cite{Hsieh2015,Wang2016,LWL:CMS:2018}. On the other hand, the same approach can be generalized to the system with multiple ion species.
\end{rmk}

\paragraph{The Poisson--Nernst--Planck--Fourier System}
Combining the above constitutive relations with the kinematic equations, we conclude with a Poisson--Nernst--Planck--Fourier system:
\begin{equation}
\begin{cases}
\displaystyle p_t + \nabla \cdot (p v^p) = 0,\quad p v^p = -D^p [\nabla (p\theta) +p \nabla \phi],\\[3mm]
\displaystyle n_t + \nabla \cdot (n v^n) = 0,\quad n v^n = -D^n [\nabla (n\theta) -n \nabla \phi],\\[3mm]
- \Delta \phi =p-n,\\[3mm]
\displaystyle (c^p p + c^n n)\theta_t + (c^p p v^p + c^n n v^n)\nabla \theta\\[3mm]
\displaystyle \qquad\qquad =  \nabla \cdot k\nabla \theta +\frac{p|v^p|^2}{D^p} +\frac{n|v^n|^2}{D^n} - p\theta \nabla \cdot v^p - n\theta \nabla \cdot v^n.
\end{cases}
\label{PNPF}
\end{equation}


It is straightforward to verify \eqref{PNPF} satisfies the conservation of mass: \eqref{ma_con}, and the First Law of Thermodynamics (conservation of energy): \eqref{energy_con} (see Appendix \ref{ap_ec} for detailed calculation). To verify the PNPF system satisfies the Second Law of Thermodynamics: \eqref{energy_law}, we investigate the local entropic density:
\begin{align*}
\frac{\partial }{\partial t}\eta &= -(\log p -c^p \log \theta +1 )p_t - (\log n -c^n \log \theta +1 )n_t +\frac{pc^p + nc^n}{\theta}\theta_t \\
&= (\log p -c^p \log \theta +1 )\nabla \cdot (p v^p) + (\log n -c^n \log \theta +1 )\nabla \cdot (n v^n) \\
&\qquad +\frac{pc^p + nc^n}{\theta}\theta_t \\
&= -\nabla \cdot [\eta^p v^p + \eta^n v^n] -\nabla \cdot \frac{q}{\theta}+ \widetilde{\triangle}.
\end{align*}
Here the first term is the transport of the entropy along with the particle velocities. The second term is the heat flux.  The last term is the entropy production rate. This is the differential form of \eqref{energy_law}, which is equivalent to the Clausius--Duhem inequality. 

\begin{rmk}
It should be noted that, the Clausius--Duhem inequality could take different forms. If we choose another form, e.g., $\eta_t + \nabla \cdot \frac{q}{\theta}= \widetilde{\triangle}\geq 0$, there will be another set of closed PDE system, which differs from the force balance given by Onsager's principle, but satisfies the same energy law in \eqref{energy_law}. This is because the solution to \eqref{balance} might not be unique.
\end{rmk}

To justify the validity of the PNPF model, we show that \eqref{PNPF} also satisfies the Onsager's reciprocal relation. Rewrite the particle and energy fluxes into the form of
\begin{equation*}
\begin{cases}
\displaystyle j^p =  -L_{pp} \nabla \frac{\mu_p}{\theta} - L_{pn} \nabla\frac{\mu_n}{\theta}+L_{p\theta} \nabla\frac{1}{\theta},\\[3mm]
\displaystyle j^n = - L_{np} \nabla \frac{\mu_p}{\theta} - L_{nn} \nabla\frac{\mu_n}{\theta}+L_{n\theta} \nabla\frac{1}{\theta},\\[3mm]
\displaystyle j^e = - L_{\theta p} \nabla \frac{\mu_p}{\theta} - L_{\theta n} \nabla\frac{\mu_n}{\theta}+L_{\theta\theta} \nabla\frac{1}{\theta}.
\end{cases}
\end{equation*}
Here $\mu_p = \theta(\log p -c^p\log\theta) +\phi$ and $\mu_n = \theta(\log n -c^n \log \theta)-\phi$ are the  chemical potentials of the two ion species. $L_{ij}$ are the coefficients describing the ratio between flows and ``forces''. Onsager's reciprocal relation requires these coefficients to be symmetric, namely, $L_{ij}=L_{ji}$. In the consititutive relations of \eqref{PNPF},
we can see,
\begin{equation*}
\begin{cases}
L_{pn}=L_{np}=0,\\[3mm]
L_{p\theta}=L_{\theta p} = D^p p \theta [(c^p+1)\theta + \phi],\\[3mm]
L_{n\theta}=L_{\theta n} = D^n n \theta [(c^n+1)\theta - \phi].
\end{cases}
\end{equation*}

\section{Well-Posedness Result}

In this section, we deal with the global well-posed-ness of the PNPF system (\ref{PNPF}) with the initial data sufficiently close to constant equilibrium states. For simplicity, we assume $c^p = c^n = c > 1$ and $D^p = D^n = \epsilon = k = 1$. Then the system (\ref{PNPF}) can be written as
\begin{align}
n_t &= \nabla \cdot \left( \nabla(n \theta) - n \nabla \phi \right), \label{original_eqn_1}\\
p_t &= \nabla \cdot \left( \nabla(p \theta) + p \nabla \phi \right), \label{original_eqn_2}\\
\Delta \phi &= n - p, \quad \lim_{|x| \rightarrow \infty} \phi = 0, \label{original_eqn_3}\\
c [(n+p) &\theta_t - (\nabla(n \theta) - n \nabla \phi + \nabla(p \theta) + p \nabla \phi) \cdot \nabla \theta] \notag\\
\quad &= n \theta \nabla \cdot \left( \dfrac{1}{n} ( \nabla(n \theta) - n \nabla \phi) \right) + p \theta \nabla \cdot \left( \dfrac{1}{p} ( \nabla(p \theta) + p \nabla \phi) \right) \label{original_eqn_4}\\
&\quad\quad + \Delta \theta + \dfrac{1}{n} \left| \nabla(n \theta) - n \nabla \phi \right|^2 + \dfrac{1}{p} \left| \nabla(p \theta) + p \nabla \phi \right|^2, \notag
\end{align}
for $x \in \mathbb{R}^3, t > 0$. By (\ref{original_eqn_3}), $\phi$ can be expressed in terms of $n$ and $p$ as
\begin{align}\label{original_eqn_3_1}
\phi(x,t) = -\frac{1}{4\pi} \int_{\mathbb{R}^3} \frac{n(y,t) - p(y,t)}{|x - y|} dy.
\end{align}
We look for solutions $(n,p,\phi,\theta)$ of system (\ref{original_eqn_1})--(\ref{original_eqn_4}) with initial data
\begin{align}\label{original_IC}
(n,p,\theta)(x,0) = (n_0,p_0,\theta_0)(x),
\end{align}
where $n_0$, $p_0$ and $\theta_0$ are close to constant states. As for the initial data for $\phi$, it can be determined by (\ref{original_eqn_3_1}) and (\ref{original_IC})
\begin{align*}
\phi(x,0) = \phi_0(x) := -\frac{1}{4\pi} \int_{\mathbb{R}^3} \frac{n_0(y) - p_0(y)}{|x - y|} dy.
\end{align*}
Without loss of generality, we assume that $n_0$, $p_0$ and $\theta_0$ are close to $1$, and state our result as follows.

\begin{thm}\label{main_theorem}
Suppose that
\begin{align*}
n_0 -1 , p_0 - 1 , \theta_0 - 1 \in H^2(\mathbb{R}^3), \quad \nabla \phi_0 \in L^2(\mathbb{R}^3).
\end{align*}
There is a positive constant $\delta_0$ such that if
\begin{align}\label{assumption_delta_0}
\| (n_0 -1 , p_0 - 1 , \theta_0 - 1) \|_{H^2} + \| \nabla \phi_0 \|_{L^2} \leq \delta_0,
\end{align}
then the system (\ref{original_eqn_1})--(\ref{original_eqn_4}) with (\ref{original_IC}) has a unique global solution.
\end{thm}

\begin{rmk}
If $c=1$, the temperature equation can be further simplified. However, the specific heat capacitance $c=3/2$ for monoatomic ideal gas, $c=5/2$ for diatomic rigid molecule, and $c=3$ for multi-atomic rigid molecule. In general, this constant is always greater than 1.
\end{rmk}

In the following, we will use $C$ to denote the generic constant that may vary from line to line. We use the notation for the norms $\| \cdot \|_{L^p} = \| \cdot \|_{L^p(\mathbb{R}^3)}$ and $\| \cdot \|_{H^k} = \| \cdot \|_{H^k(\mathbb{R}^3)}$. And $\langle \cdot , \cdot \rangle$ denotes the inner product in $L^2(\mathbb{R}^3)$. $\partial_i = \partial_{x_i}$ and $\partial_{ij} = \partial_{x_i} \partial_{x_j}$ for $i, j \in \{ 1, 2, 3 \}$.

\section{Proof of Theorem \ref{main_theorem}}

Now we are going to prove Theorem \ref{main_theorem}. In the framework of the energy method, the global existence can be obtained by the local well-posedness and global in time a priori estimates. Notice that the system (\ref{PNPF}) is strictly parabolic provided that $n$, $p$ and $\theta$ are positive away from zero. Therefore, under the assumption (\ref{assumption_delta_0}), the existence and uniqueness of local strong solutions of the system (\ref{PNPF}) can be proved in a standard way using the Galerkin method and the fixed-point argument. Therefore the proof is omitted and referred to, for example, Chapter VII of \cite{LSU68}.

Before we prove the a priori estimates, we first reformulate the equation as follows. Denoting $u = n + p$ and $v = n - p$, (\ref{original_eqn_1})--(\ref{original_IC}) can be rewritten as
\begin{align}
u_t &= \theta \Delta u + 2 \nabla \theta \cdot \nabla u + u \Delta \theta - \nabla v \cdot \nabla \phi - v^2, \label{eqn_0_1}\\
v_t &= \theta \Delta v + 2 \nabla \theta \cdot \nabla v + v \Delta \theta - \nabla u \cdot \nabla \phi - uv, \label{eqn_0_2}\\
c\theta_t &= \left(\theta + \dfrac{1}{u} \right) \Delta \theta + \dfrac{\theta^2}{u} \Delta u + \left(c + 1\right) |\nabla \theta|^2 \label{eqn_0_3}\\
&\quad\quad + \left(c + 3\right) \dfrac{\theta}{u} \nabla u \cdot \nabla \theta - 2\ \dfrac{\theta}{u} \nabla v \cdot \nabla \phi - \left(c + 2\right) \dfrac{v}{u} \nabla \theta \cdot \nabla \phi - \dfrac{\theta v^2}{u} + |\nabla \phi|^2. \notag
\end{align}
Here, in (\ref{eqn_0_1})--(\ref{eqn_0_3}), $\phi$ satisfies (\ref{original_eqn_3_1}), which can be rewritten in terms of $v$ as
\begin{align*}
\phi(x,t) = -\dfrac{1}{4\pi} \int_{\mathbb{R}^3} \dfrac{v(y,t)}{|x-y|} dy.
\end{align*}
Letting $\tilde{u} = u - 2$ and $\tilde{\theta} = \theta - 1$, we further reformulate (\ref{eqn_0_1})--(\ref{eqn_0_3}) as
\begin{align}
\tilde{u}_t &= \Delta \tilde{u} + 2 \Delta \tilde{\theta} + \tilde{\theta} \Delta \tilde{u} + \tilde{u} \Delta \tilde{\theta} + 2 \nabla \tilde{\theta} \cdot \nabla \tilde{u} - \nabla v \cdot \nabla \phi - v^2, \label{eqn_1_1}\\
v_t &= \Delta v - 2v + \tilde{\theta} \Delta v + v \Delta \tilde{\theta} + 2 \nabla \tilde{\theta} \cdot \nabla v - \nabla \tilde{u} \cdot \nabla \phi - \tilde{u} v, \label{eqn_1_2} \\
c\tilde{\theta}_t &= \dfrac{3}{2} \Delta \tilde{\theta} + \dfrac{1}{2} \Delta \tilde{u} + \tilde{\theta} \Delta \tilde{\theta} - \dfrac{\tilde{u}}{2(2 + \tilde{u})} \Delta \tilde{\theta} + \dfrac{2 \tilde{\theta}^2 + 4 \tilde{\theta} - \tilde{u}}{2(2 + \tilde{u})} \Delta \tilde{u} \label{eqn_1_3}\\
&\quad\quad + \left(c + 1\right) |\nabla \tilde{\theta}|^2 + \left(c + 3\right) \dfrac{1 + \tilde{\theta}}{2 + \tilde{u}} \nabla \tilde{u} \cdot \nabla \tilde{\theta} \notag\\
&\quad\quad - 2\ \dfrac{1 + \tilde{\theta}}{2 + \tilde{u}} \nabla v \cdot \nabla \phi - \left(c + 2\right) \dfrac{v}{2 + \tilde{u}} \nabla \tilde{\theta} \cdot \nabla \phi - \dfrac{(1 + \tilde{\theta}) v^2}{2 + \tilde{u}} + |\nabla \phi|^2. \notag
\end{align}
If we a priorily assume that
\begin{align*}
\| (n - 1 , p - 1 , \theta - 1) \|_{H^2} + \| \nabla \phi \|_{L^2} \leq \delta_0,
\end{align*}
then equivalently we have
\begin{align}\label{H2_small}
\| (\tilde{u} , v, \tilde{\theta}) \|_{H^2} + \| \nabla \phi \|_{L^2} \leq \delta
\end{align}
for some $0 < \delta \sim \delta_0$. Moreover, by the  Sobolev embedding, (\ref{H2_small}) implies that
\begin{align*}
\| (\tilde{u} , v, \tilde{\theta}) \|_{L^\infty} \leq C \delta.
\end{align*}
In order to prove Theorem \ref{main_theorem}, it suffices to show the following a priori estimate.

\begin{prop}\label{prop_H2_est}
Assuming (\ref{H2_small}) for $\delta$ sufficiently small, we have
\begin{align*}
&\dfrac{1}{2} \dfrac{d}{dt} \left( \| (\tilde{u} , v) \|_{H^2}^2 + 2c \| \tilde{\theta} \|_{H^2}^2 + \| \nabla \phi \|_{L^2}^2 \right) \\
&\qquad + C_1 \| \nabla (\tilde{u} , v, \tilde{\theta}) \|_{H^2}^2 + C_2 \| v \|_{H^2}^2 + C_3 \| \nabla \phi \|_{L^2}^2 \leq 0.
\end{align*}
\end{prop}

\noindent With the aid of Proposition \ref{prop_H2_est}, we conclude that
\begin{align*}
\| (\tilde{u} , v) \|_{H^2}^2 + 2c \| \tilde{\theta} \|_{H^2}^2 + \| \nabla \phi \|_{L^2}^2 &\leq \| (\tilde{u}_0 , v_0) \|_{H^2}^2 + 2c \| \tilde{\theta}_0 \|_{H^2}^2 + \| \nabla \phi_0 \|_{L^2}^2 \\
&\leq C \delta^2
\end{align*}
for any $t \geq 0$, where $\tilde{u}_0 = n_0 + p_0 - 2$, $v_0 = n_0 - p_0$, and $\tilde{\theta}_0 = \theta - 1$. Theorem \ref{main_theorem} then follows as we noted in the beginning of this section.

In the remaining of this section, we will prove Proposition \ref{prop_H2_est}. The proof can be separated into several lemmas. First, we prove the $L^2$ estimate as follows.

\begin{lem}\label{lem_L2_est}
Assuming (\ref{H2_small}) for $\delta$ sufficiently small, we have
\begin{align*}
\dfrac{1}{2} \dfrac{d}{dt} \left( \| (\tilde{u} , v) \|_{L^2}^2 + 2c \| \tilde{\theta} \|_{L^2}^2 \right) + C \| \nabla (\tilde{u} , v, \tilde{\theta}) \|_{L^2}^2 + C \| v \|_{L^2}^2 \leq C \delta \| \nabla \phi \|_{L^2}^2.
\end{align*}
\end{lem}

\begin{proof}
Multiplying the equations (\ref{eqn_1_1}), (\ref{eqn_1_2}) and (\ref{eqn_1_3}) by $\tilde{u}$, $v$ and $2 \tilde{\theta}$, respectively, integrating over $\mathbb{R}^3$, doing integration by parts, and then summing the resulting equalities, we obtain
\begin{align}\label{id_1_01}
&\quad\ \dfrac{1}{2} \dfrac{d}{dt} \left( \| (\tilde{u} , v) \|_{L^2}^2 + 2c \| \tilde{\theta} \|_{L^2}^2 \right) + \| \nabla (\tilde{u},v) \|_{L^2}^2 + 3 \| \nabla \tilde{\theta} \|_{L^2}^2 + 2 \| v \|_{L^2}^2 \notag\\
&= 2 \langle \Delta \tilde{\theta} , \tilde{u} \rangle + \langle \Delta \tilde{u} , \tilde{\theta} \rangle + \langle \tilde{\theta} \Delta \tilde{u} + \tilde{u} \Delta \tilde{\theta} , \tilde{u} \rangle + \langle \tilde{\theta} \Delta v + v \Delta \tilde{\theta} , v \rangle \\
&\quad\quad + \left\langle 2 \tilde{\theta} \Delta \tilde{\theta} - \dfrac{\tilde{u}}{2 + \tilde{u}} \Delta \tilde{\theta} + \dfrac{2 \tilde{\theta}^2 + 4 \tilde{\theta} - \tilde{u}}{2 + \tilde{u}} \Delta \tilde{u} , \tilde{\theta} \right\rangle \notag\\
&\quad\quad + \left\langle 2 \nabla \tilde{\theta} \cdot \nabla \tilde{u} - \nabla v \cdot \nabla \phi - v^2 , \tilde{u} \right\rangle + \left\langle 2 \nabla \tilde{\theta} \cdot \nabla v - \nabla \tilde{u} \cdot \nabla \phi - \tilde{u} v , v \right\rangle \notag\\
&\quad\quad + 2 \left\langle \left(c + 1\right) |\nabla \tilde{\theta}|^2 + \left(c + 3\right) \dfrac{1 + \tilde{\theta}}{2 + \tilde{u}} \nabla \tilde{u} \cdot \nabla \tilde{\theta} , \tilde{\theta} \right\rangle \notag\\
&\quad\quad - 2 \left\langle 2\ \dfrac{1 + \tilde{\theta}}{2 + \tilde{u}} \nabla v \cdot \nabla \phi + \left(c + 2\right) \dfrac{v}{2 + \tilde{u}} \nabla \tilde{\theta} \cdot \nabla \phi + \dfrac{(1 + \tilde{\theta}) v^2}{2 + \tilde{u}} - |\nabla \phi|^2 , \tilde{\theta} \right\rangle. \notag
\end{align}
We can estimate the right-hand side of the last identity term by term as follows. First, by doing integration by parts and Young's inequality, it holds
\begin{align}\label{id_1_02}
\left| 2 \langle \Delta \tilde{\theta} , \tilde{u} \rangle + \langle \Delta \tilde{u} , \tilde{\theta} \rangle \right| &= 3 \left| \langle \nabla \tilde{u} , \nabla \tilde{\theta} \rangle \right| \notag\\
&\leq \dfrac{9}{11} \| \nabla \tilde{u} \|_{L^2}^2 +  \dfrac{11}{4} \| \nabla \tilde{\theta} \|_{L^2}^2.
\end{align}
Again doing integration by parts, and using H\"{o}lder's inequality and Young's inequality, we have
\begin{align}\label{id_1_03}
| \langle \tilde{\theta} \Delta \tilde{u} + \tilde{u} \Delta \tilde{\theta} , \tilde{u} \rangle | &\leq 3 | \langle \nabla \tilde{\theta} \cdot \nabla \tilde{u} , \tilde{u} \rangle | + | \langle | \nabla \tilde{u} |^2 , \tilde{\theta} \rangle | \notag\\
&\leq 3 \| \nabla \tilde{\theta} \|_{L^2} \| \nabla \tilde{u} \|_{L^2} \| \tilde{u} \|_{L^\infty} + \| \nabla \tilde{u} \|_{L^2}^2 \| \tilde{\theta} \|_{L^\infty} \notag\\
&\leq C \delta \| \nabla (\tilde{u} , \tilde{\theta}) \|_{L^2}^2.
\end{align}
Similarly, we have
\begin{align}\label{id_1_04}
| \langle \tilde{\theta} \Delta v + v \Delta \tilde{\theta} , v \rangle | \leq C \delta \| \nabla (v , \tilde{\theta}) \|_{L^2}^2,
\end{align}
and
\begin{align}\label{id_1_05}
\left| \left\langle 2 \tilde{\theta} \Delta \tilde{\theta} - \dfrac{\tilde{u}}{2 + \tilde{u}} \Delta \tilde{\theta} + \dfrac{2 \tilde{\theta}^2 + 4 \tilde{\theta} - \tilde{u}}{2 + \tilde{u}} \Delta \tilde{u} , \tilde{\theta} \right\rangle \right| \leq C \delta \| \nabla ( u , \tilde{\theta}) \|_{L^2}^2.
\end{align}
Next, integration by parts yields
\begin{align*}
\left\langle \nabla v \cdot \nabla \phi + v^2 , \tilde{u} \right\rangle + \left\langle \nabla \tilde{u} \cdot \nabla \phi + \tilde{u} v , v \right\rangle &= \left\langle \tilde{u} v , \Delta \phi \right\rangle + 2 \left\langle \tilde{u} , v^2 \right\rangle \\
&= 3 \left\langle \tilde{u} , v^2 \right\rangle.
\end{align*}
Thus,
\begin{align}\label{id_1_06}
&\quad\ \left| \left\langle 2 \nabla \tilde{\theta} \cdot \nabla \tilde{u} - \nabla v \cdot \nabla \phi - v^2 , \tilde{u} \right\rangle + \left\langle 2 \nabla \tilde{\theta} \cdot \nabla v - \nabla \tilde{u} \cdot \nabla \phi - \tilde{u} v , v \right\rangle \right| \notag\\
&\leq 2 \left| \left\langle \nabla \tilde{\theta} \cdot \nabla \tilde{u} , \tilde{u} \right\rangle \right| + 2 \left| \left\langle \nabla \tilde{\theta} \cdot \nabla v , v \right\rangle \right| + 3 \left| \left\langle \tilde{u} , v^2 \right\rangle \right| \notag\\
&\leq 2 \| \nabla \tilde{\theta} \|_{L^2} \| \nabla \tilde{u} \|_{L^2} \| \tilde{u} \|_{L^\infty} + 2 \| \nabla \tilde{\theta} \|_{L^2} \| \nabla v \|_{L^2} \| v \|_{L^\infty} + 3 \| u \|_{L^\infty} \| v \|_{L^2}^2 \notag\\
&\leq C \delta \| \nabla ( \tilde{u} , v , \tilde{\theta} ) \|_{L^2}^2 + C \delta \| v \|_{L^2}^2.
\end{align}
For $\delta$ sufficiently small such that $\| \tilde{u} \|_{L^\infty} < 1$, we obtain
\begin{align}\label{id_1_07}
&\quad\ \left| \left\langle \left(c + 1\right) |\nabla \tilde{\theta}|^2 + \left(c + 3\right) \dfrac{1 + \tilde{\theta}}{2 + \tilde{u}} \nabla \tilde{u} \cdot \nabla \tilde{\theta} , \tilde{\theta} \right\rangle \right| \notag\\
&\leq C \| \nabla \tilde{\theta} \|_{L^2}^2 \| \nabla \tilde{\theta} \|_{L^\infty} + C \| \nabla \tilde{u} \|_{L^2} \| \nabla \tilde{\theta} \|_{L^2}  \| \tilde{\theta} \|_{L^\infty}^2 \notag\\
&\leq C \delta \| \nabla ( \tilde{u} , \tilde{\theta} ) \|_{L^2}^2.
\end{align}
Finally, for the last term, we have
\begin{align}\label{id_1_08}
&\quad\ \left| \left\langle 2\ \dfrac{1 + \tilde{\theta}}{2 + \tilde{u}} \nabla v \cdot \nabla \phi + \left(c + 2\right) \dfrac{v}{2 + \tilde{u}} \nabla \tilde{\theta} \cdot \nabla \phi + \dfrac{(1 + \tilde{\theta}) v^2}{2 + \tilde{u}} - |\nabla \phi|^2 , \tilde{\theta} \right\rangle \right| \notag\\
&\leq C \| \nabla v \|_{L^2} \| \nabla \phi \|_{L^2} \| \tilde{\theta} \|_{L^\infty} + C \| v \|_{L^\infty} \| \nabla \tilde{\theta} \|_{L^2} \| \nabla \phi \|_{L^2} \| \tilde{\theta} \|_{L^\infty} \notag\\
&\qquad\qquad + C \| v \|_{L^2}^2 \| \tilde{\theta} \|_{L^\infty} + C \| \nabla \phi \|_{L^2}^2 \| \tilde{\theta} \|_{L^\infty} \notag\\
&\leq C \delta \| \nabla ( v , \tilde{\theta} ) \|_{L^2}^2 + C \delta \| v \|_{L^2}^2 + C \delta \| \nabla \phi \|_{L^2}^2. 
\end{align}
Combining (\ref{id_1_01})--(\ref{id_1_08}), the proof is completed since $\delta$ is sufficiently small.
\end{proof}

Next, we prove the higher order estimate.

\begin{lem}\label{lem_H1_est}
Assuming (\ref{H2_small}) for $\delta$ sufficiently small, we have
\begin{align*}
&\quad\ \dfrac{1}{2} \dfrac{d}{dt} \left( \| \nabla (\tilde{u} , v) \|_{H^1}^2 + 2c \| \nabla \tilde{\theta} \|_{H^1}^2 \right) + C \| \nabla^2 (\tilde{u} , v, \tilde{\theta}) \|_{H^1}^2 + C \| \nabla v \|_{L^2}^2 \\
&\leq C\delta \left( \| \nabla (\tilde{u} , \tilde{\theta}) \|_{L^2}^2 + \| v \|_{L^2}^2 + \| \nabla \phi \|_{L^2}^2 \right).
\end{align*}
\end{lem}

\begin{proof}
By applying $\partial_i$ to the equations (\ref{eqn_1_1}), (\ref{eqn_1_2}) and (\ref{eqn_1_3}) for $i = 1, 2, 3$, multiplying the resulting equations by $\partial_i \tilde{u}$, $\partial_i v$ and $2 \partial_i \tilde{\theta}$, respectively, integrating over $\mathbb{R}^3$, doing integration by parts, and then summing the resulting equalities, we obtain
\begin{align}\label{id_1_09}
&\quad\ \dfrac{1}{2} \dfrac{d}{dt} \left( \| (\partial_i \tilde{u} , \partial_i v) \|_{L^2}^2 + 2c \| \partial_i \tilde{\theta} \|_{L^2}^2 \right) \notag\\
&\qquad\qquad + \| \nabla (\partial_i \tilde{u} , \partial_i v) \|_{L^2}^2 + 3 \| \nabla ( \partial_i \tilde{\theta} ) \|_{L^2}^2 + 2 \| \partial_i v \|_{L^2}^2 \notag\\
&= 2 \langle \partial_i ( \Delta \tilde{\theta} ) , \partial_i \tilde{u} \rangle + \langle \partial_i ( \Delta \tilde{u} ) , \partial_i \tilde{\theta} \rangle \\
&\qquad\qquad + \langle \partial_i ( \tilde{\theta} \Delta \tilde{u} + \tilde{u} \Delta \tilde{\theta} ) , \partial_i \tilde{u} \rangle + \langle \partial_i ( \tilde{\theta} \Delta v + v \Delta \tilde{\theta} ) , \partial_i v \rangle \notag\\
&\qquad\qquad + \left\langle \partial_i \left( 2 \tilde{\theta} \Delta \tilde{\theta} - \dfrac{\tilde{u}}{2 + \tilde{u}} \Delta \tilde{\theta} + \dfrac{2 \tilde{\theta}^2 + 4 \tilde{\theta} - \tilde{u}}{2 + \tilde{u}} \Delta \tilde{u} \right) , \partial_i \tilde{\theta} \right\rangle \notag\\
&\qquad\qquad + \left\langle \partial_i \left( 2 \nabla \tilde{\theta} \cdot \nabla \tilde{u} - \nabla v \cdot \nabla \phi - v^2 \right) , \partial_i \tilde{u} \right\rangle \notag\\
&\qquad\qquad + \left\langle \partial_i \left( 2 \nabla \tilde{\theta} \cdot \nabla v - \nabla \tilde{u} \cdot \nabla \phi - \tilde{u} v \right) , \partial_i v \right\rangle \notag\\
&\qquad\qquad + 2 \left\langle \partial_i \left( \left(c + 1\right) |\nabla \tilde{\theta}|^2 + \left(c + 3\right) \dfrac{1 + \tilde{\theta}}{2 + \tilde{u}} \nabla \tilde{u} \cdot \nabla \tilde{\theta} \right) , \partial_i \tilde{\theta} \right\rangle \notag\\
&\qquad\qquad - 2 \Bigg\langle \partial_i \Bigg( 2\ \dfrac{1 + \tilde{\theta}}{2 + \tilde{u}} \nabla v \cdot \nabla \phi \notag\\
&\qquad\qquad\qquad\qquad\qquad + \left(c + 2\right) \dfrac{v}{2 + \tilde{u}} \nabla \tilde{\theta} \cdot \nabla \phi + \dfrac{(1 + \tilde{\theta}) v^2}{2 + \tilde{u}} - |\nabla \phi|^2 \Bigg) , \partial_i \tilde{\theta} \Bigg\rangle. \notag
\end{align}
We now estimate the right-hand side of (\ref{id_1_09}) as follows. First, doing integration by parts and using Young’s inequality yield
\begin{align}\label{id_1_10}
\left| 2 \langle \partial_i ( \Delta \tilde{\theta} ) , \partial_i \tilde{u} \rangle + \langle \partial_i ( \Delta \tilde{u} ) , \partial_i \tilde{\theta} \rangle \right| &= 3 \left| \langle \nabla ( \partial_i \tilde{u} ) , \nabla ( \partial_i \tilde{\theta} ) \rangle \right| \notag \\
&\leq \frac{9}{11} \| \nabla ( \partial_i \tilde{u} ) \|_{L^2}^2 + \frac{11}{4} \| \nabla ( \partial_i \tilde{\theta} ) \|_{L^2}^2. 
\end{align}
Next, by doing integration by parts, and using H\"{o}lder's inequality, Sobolev inequality and Young's inequality, we have
\begin{align}\label{id_1_11}
&\quad\ \left| \langle \partial_i ( \tilde{\theta} \Delta \tilde{u} + \tilde{u} \Delta \tilde{\theta} ) , \partial_i \tilde{u} \rangle \right| \notag\\
&\leq \left| \langle \partial_i \tilde{\theta} \Delta \tilde{u} , \partial_i \tilde{u} \rangle \right| + \left| \langle \tilde{\theta} \Delta ( \partial_i \tilde{u} ) , \partial_i \tilde{u} \rangle \right| + \left| \langle \partial_i \tilde{u} \Delta \tilde{\theta} , \partial_i \tilde{u} \rangle \right| + \left| \langle \tilde{u} \Delta ( \partial_i \tilde{\theta} ) , \partial_i \tilde{u} \rangle \right| \notag\\
&\leq \left| \langle \partial_i \tilde{\theta} \Delta \tilde{u} , \partial_i \tilde{u} \rangle \right| + \left| \langle | \nabla (\partial_i \tilde{u}) |^2 , \tilde{\theta} \rangle \right| + \left| \langle \nabla \tilde{\theta} \cdot \nabla ( \partial_i \tilde{u} ) , \partial_i \tilde{u} \rangle \right| \notag\\
&\qquad\qquad + \left| \langle \partial_i \tilde{u} \Delta \tilde{\theta} , \partial_i \tilde{u} \rangle \right| + \left| \langle \nabla \tilde{u} \cdot \nabla ( \partial_i \tilde{\theta} ) , \partial_i \tilde{u} \rangle \right| + \left| \langle \nabla ( \partial_i \tilde{u} ) \cdot \nabla ( \partial_i \tilde{\theta} ) , \tilde{u} \rangle \right| \notag\\
&\leq 2 \| \nabla^2 \tilde{u} \|_{L^2} \| \nabla \tilde{\theta} \|_{L^4} \| \nabla \tilde{u} \|_{L^4} + \| \tilde{\theta} \|_{L^\infty} \| \nabla^2 \tilde{u} \|_{L^2}^2 \notag\\
&\qquad\qquad + 2 \| \nabla^2 \tilde{\theta} \|_{L^2} \| \nabla \tilde{u} \|_{L^4}^2 + \| \tilde{u} \|_{L^\infty} \| \nabla^2 \tilde{u} \|_{L^2} \| \nabla^2 \tilde{\theta} \|_{L^2} \notag\\
&\leq C \| \tilde{u} \|_{H^2} \| \nabla \tilde{\theta} \|_{H^1} \| \nabla \tilde{u} \|_{H^1} + \| \tilde{\theta} \|_{L^\infty} \| \nabla^2 \tilde{u} \|_{L^2}^2 \notag\\
&\qquad\qquad + C \| \tilde{\theta} \|_{H^2} \| \nabla \tilde{u} \|_{H^1}^2 + \| \tilde{u} \|_{L^\infty} \| \nabla^2 \tilde{u} \|_{L^2} \| \nabla^2 \tilde{\theta} \|_{L^2} \notag\\
&\leq C\delta \| \nabla (\tilde{u} , \tilde{\theta}) \|_{H^1}^2.
\end{align}
Similarly,
\begin{align}\label{id_1_12}
\left| \langle \partial_i ( \tilde{\theta} \Delta v + v \Delta \tilde{\theta} ) , \partial_i v \rangle \right| \leq C\delta \| \nabla (v , \tilde{\theta}) \|_{H^1}^2,
\end{align}
and, for $\delta$ sufficiently small such that  $\| \tilde{u} \|_{L^\infty} < 1$,
\begin{align}\label{id_1_13}
\left| \left\langle \partial_i \left( 2 \tilde{\theta} \Delta \tilde{\theta} - \dfrac{\tilde{u}}{2 + \tilde{u}} \Delta \tilde{\theta} + \dfrac{2 \tilde{\theta}^2 + 4 \tilde{\theta} - \tilde{u}}{2 + \tilde{u}} \Delta \tilde{u} \right) , \partial_i \tilde{\theta} \right\rangle \right| \leq C\delta \| \nabla (\tilde{u} , \tilde{\theta}) \|_{H^1}^2.
\end{align}
Notice that the Hardy-Littlewood-Sobolev inequality gives
\begin{align}\label{HLS_ineq}
\| \nabla^k \phi \|_{L^6} \leq C \| \nabla^{k-1} v \|_{L^2}
\end{align}
for each $k \in \mathbb{N}$. Together with H\"{o}lder's inequality, Sobolev inequality and Young's inequality, it is easy to deduce that
\begin{align}\label{id_1_14}
&\quad\ \left| \left\langle \partial_i \left( 2 \nabla \tilde{\theta} \cdot \nabla \tilde{u} - \nabla v \cdot \nabla \phi - v^2 \right) , \partial_i \tilde{u} \right\rangle \right| \notag\\
&\leq 2 \left| \left\langle \nabla ( \partial_i \tilde{\theta} ) \cdot \nabla \tilde{u} , \partial_i \tilde{u} \right\rangle \right| + 2 \left| \left\langle \nabla \tilde{\theta} \cdot \nabla ( \partial_i \tilde{u} ) , \partial_i \tilde{u} \right\rangle \right| \notag\\
&\qquad\qquad + \left| \left\langle \nabla (\partial_i v ) \cdot \nabla \phi , \partial_i \tilde{u} \right\rangle \right| + \left| \left\langle \nabla v \cdot \nabla (\partial_i \phi ), \partial_i \tilde{u} \right\rangle \right| + 2 \left| \left\langle v \partial_i v , \partial_i \tilde{u} \right\rangle \right| \notag\\
&\leq 2 \| \nabla^2 \tilde{\theta} \|_{L^2} \| \nabla \tilde{u} \|_{L^4}^2 + 2 \| \nabla^2 \tilde{u} \|_{L^2} \| \nabla \tilde{\theta} \|_{L^4} \| \nabla \tilde{u} \|_{L^4} \notag\\
&\qquad\qquad + \| \nabla^2 v \|_{L^2} \| \nabla \tilde{u} \|_{L^3} \| \nabla \phi \|_{L^6} + \| \nabla \tilde{u} \|_{L^2} \| \nabla v \|_{L^3} \| \nabla^2 \phi \|_{L^6} \notag\\
&\qquad\qquad + 2 \| v \|_{L^\infty} \| \nabla v \|_{L^2} \| \nabla \tilde{u} \|_{L^2} \notag\\
&\leq C \| \tilde{\theta} \|_{H^2} \| \nabla \tilde{u} \|_{H^1}^2 + C \| \tilde{u} \|_{H^2} \| \nabla \tilde{\theta} \|_{H^1} \| \nabla \tilde{u} \|_{H^1} \notag\\
&\qquad\qquad + C \| v \|_{H^2} \| \nabla \tilde{u} \|_{H^1} \| v \|_{L^2} + C \| \tilde{u} \|_{H^2} \| \nabla v \|_{H^1} \| \nabla v \|_{L^2} \notag\\
&\qquad\qquad + 2 \| v \|_{L^\infty} \| \nabla v \|_{L^2} \| \nabla \tilde{u} \|_{L^2} \notag\\
&\leq C\delta \left( \| \nabla (\tilde{u} , v , \tilde{\theta}) \|_{H^1}^2 + \| v \|_{L^2}^2 \right).
\end{align}
Similarly, we have
\begin{align}\label{id_1_15}
\left| \left\langle \partial_i \left( 2 \nabla \tilde{\theta} \cdot \nabla v - \nabla \tilde{u} \cdot \nabla \phi - \tilde{u} v \right) , \partial_i v \right\rangle \right| \leq C\delta \left( \| \nabla (\tilde{u} , v , \tilde{\theta}) \|_{H^1}^2 + \| v \|_{L^2}^2 \right),
\end{align}
\begin{align}\label{id_1_16}
\left| \left\langle \partial_i \left( \left(c + 1\right) |\nabla \tilde{\theta}|^2 + \left(c + 3\right) \dfrac{1 + \tilde{\theta}}{2 + \tilde{u}} \nabla \tilde{u} \cdot \nabla \tilde{\theta} \right) , \partial_i \tilde{\theta} \right\rangle \right| \leq C\delta \| \nabla (\tilde{u} , \tilde{\theta}) \|_{H^1}^2,
\end{align}
and
\begin{align}\label{id_1_17}
&\quad\ \left| \left\langle \partial_i \left( 2\ \dfrac{1 + \tilde{\theta}}{2 + \tilde{u}} \nabla v \cdot \nabla \phi + \left(c + 2\right) \dfrac{v}{2 + \tilde{u}} \nabla \tilde{\theta} \cdot \nabla \phi + \dfrac{(1 + \tilde{\theta}) v^2}{2 + \tilde{u}} - |\nabla \phi|^2 \right) , \partial_i \tilde{\theta} \right\rangle \right| \notag\\
&\leq C\delta \left( \| \nabla (\tilde{u} , v , \tilde{\theta}) \|_{H^1}^2 + \| v \|_{L^2}^2 + \| \nabla \phi \|_{L^2}^2 \right).
\end{align}
Combining (\ref{id_1_09})--(\ref{id_1_13}) and (\ref{id_1_14})--(\ref{id_1_17}), and summing over $i = 1, 2, 3$, we conclude
\begin{align}\label{id_1_18}
&\quad\ \dfrac{1}{2} \dfrac{d}{dt} \left( \| \nabla (\tilde{u} , v) \|_{L^2}^2 + 2c \| \nabla \tilde{\theta} \|_{L^2}^2 \right) + C \| \nabla^2 (\tilde{u} , v, \tilde{\theta}) \|_{L^2}^2 + C \| \nabla v \|_{L^2}^2 \notag\\
&\leq C\delta \left( \| \nabla (\tilde{u} , v , \tilde{\theta}) \|_{L^2}^2 + \| v \|_{L^2}^2 + \| \nabla \phi \|_{L^2}^2 \right)
\end{align}
for $\delta$ sufficiently small.

Next, applying $\partial_{ij}$ to the equations (\ref{eqn_1_1}), (\ref{eqn_1_2}) and (\ref{eqn_1_3}) for $i, j = 1, 2, 3$, multiplying the resulting equations by $\partial_{ij} \tilde{u}$, $\partial_{ij} v$ and $2 \partial_{ij} \tilde{\theta}$, respectively, integrating over $\mathbb{R}^3$, doing integration by parts, and then summing the resulting equalities, we obtain
\begin{align}\label{id_1_19}
&\quad\ \dfrac{1}{2} \dfrac{d}{dt} \left( \| (\partial_{ij} \tilde{u} , \partial_{ij} v) \|_{L^2}^2 + 2c \| \partial_{ij} \tilde{\theta} \|_{L^2}^2 \right) \notag\\
&\qquad\qquad + \| \nabla (\partial_{ij} \tilde{u} , \partial_{ij} v) \|_{L^2}^2 + 3 \| \nabla ( \partial_{ij} \tilde{\theta} ) \|_{L^2}^2 + 2 \| \partial_{ij} v \|_{L^2}^2 \notag\\
&= 2 \langle \partial_{ij} ( \Delta \tilde{\theta} ) , \partial_{ij} \tilde{u} \rangle + \langle \partial_{ij} ( \Delta \tilde{u} ) , \partial_{ij} \tilde{\theta} \rangle \\
&\qquad\qquad + \langle \partial_{ij} ( \tilde{\theta} \Delta \tilde{u} + \tilde{u} \Delta \tilde{\theta} ) , \partial_{ij} \tilde{u} \rangle + \langle \partial_i ( \tilde{\theta} \Delta v + v \Delta \tilde{\theta} ) , \partial_{ij} v \rangle \notag\\
&\qquad\qquad + \left\langle \partial_{ij} \left( 2 \tilde{\theta} \Delta \tilde{\theta} - \dfrac{\tilde{u}}{2 + \tilde{u}} \Delta \tilde{\theta} + \dfrac{2 \tilde{\theta}^2 + 4 \tilde{\theta} - \tilde{u}}{2 + \tilde{u}} \Delta \tilde{u} \right) , \partial_{ij} \tilde{\theta} \right\rangle \notag\\
&\qquad\qquad + \left\langle \partial_{ij} \left( 2 \nabla \tilde{\theta} \cdot \nabla \tilde{u} - \nabla v \cdot \nabla \phi - v^2 \right) , \partial_{ij} \tilde{u} \right\rangle \notag\\
&\qquad\qquad + \left\langle \partial_{ij} \left( 2 \nabla \tilde{\theta} \cdot \nabla v - \nabla \tilde{u} \cdot \nabla \phi - \tilde{u} v \right) , \partial_{ij} v \right\rangle \notag\\
&\qquad\qquad + 2 \left\langle \partial_{ij} \left( \left(c + 1\right) |\nabla \tilde{\theta}|^2 + \left(c + 3\right) \dfrac{1 + \tilde{\theta}}{2 + \tilde{u}} \nabla \tilde{u} \cdot \nabla \tilde{\theta} \right) , \partial_{ij} \tilde{\theta} \right\rangle \notag\\
&\qquad\qquad - 2 \Bigg\langle \partial_{ij} \Bigg( 2\ \dfrac{1 + \tilde{\theta}}{2 + \tilde{u}} \nabla v \cdot \nabla \phi \notag\\
&\qquad\qquad\qquad\qquad\qquad + \left(c + 2\right) \dfrac{v}{2 + \tilde{u}} \nabla \tilde{\theta} \cdot \nabla \phi + \dfrac{(1 + \tilde{\theta}) v^2}{2 + \tilde{u}} - |\nabla \phi|^2 \Bigg) , \partial_{ij} \tilde{\theta} \Bigg\rangle. \notag
\end{align}
We estimate the right-hand side of (\ref{id_1_19}) term by term in a way similar to the estimate for (\ref{id_1_09}). By integration by parts and Young's inequality,
\begin{align}\label{id_1_20}
\left| 2 \langle \partial_{ij} ( \Delta \tilde{\theta} ) , \partial_{ij} \tilde{u} \rangle + \langle \partial_{ij} ( \Delta \tilde{u} ) , \partial_{ij} \tilde{\theta} \rangle \right| &= 3 \left| \langle \nabla ( \partial_{ij} \tilde{u} ) , \nabla ( \partial_{ij} \tilde{\theta} ) \rangle \right| \notag \\
&\leq \frac{9}{11} \| \nabla ( \partial_{ij} \tilde{u} ) \|_{L^2}^2 + \frac{11}{4} \| \nabla ( \partial_{ij} \tilde{\theta} ) \|_{L^2}^2. 
\end{align}
Then, by integration by parts, H\"{o}lder's inequality, Sobolev inequality and Young's inequality, we have
\begin{align}\label{id_1_21}
&\quad\ \left| \langle \partial_{ij} ( \tilde{\theta} \Delta \tilde{u} + \tilde{u} \Delta \tilde{\theta} ) , \partial_{ij} \tilde{u} \rangle \right| \notag\\
&\leq \left| \langle \partial_{ij} \tilde{\theta} \Delta \tilde{u} , \partial_{ij} \tilde{u} \rangle \right| + \left| \langle \partial_j \tilde{\theta} \Delta (\partial_i \tilde{u}) , \partial_{ij} \tilde{u} \rangle \right| + \left| \langle \partial_i \tilde{\theta} \Delta (\partial_j \tilde{u}) , \partial_{ij} \tilde{u} \rangle \right| \notag\\
&\qquad\qquad + \left| \langle \tilde{\theta} \Delta (\partial_{ij} \tilde{u}) , \partial_{ij} \tilde{u} \rangle \right| + \left| \langle \partial_{ij} \tilde{u} \Delta \tilde{\theta} , \partial_{ij} \tilde{u} \rangle \right| + \left| \langle \partial_j \tilde{u} \Delta (\partial_i \tilde{\theta}) , \partial_{ij} \tilde{u} \rangle \right| \notag\\
&\qquad\qquad + \left| \langle \partial_i \tilde{u} \Delta (\partial_j \tilde{\theta}) , \partial_{ij} \tilde{u} \rangle \right| + \left| \langle \tilde{u} \Delta (\partial_{ij} \tilde{\theta}) , \partial_{ij} \tilde{u} \rangle \right| \notag\\
&\leq \left| \langle \partial_{ij} \tilde{\theta} \Delta \tilde{u} , \partial_{ij} \tilde{u} \rangle \right| + \left| \langle \partial_j \tilde{\theta} \Delta (\partial_i \tilde{u}) , \partial_{ij} \tilde{u} \rangle \right| + \left| \langle \partial_i \tilde{\theta} \Delta (\partial_j \tilde{u}) , \partial_{ij} \tilde{u} \rangle \right| \notag\\
&\qquad\qquad + \left| \langle \nabla \tilde{\theta} \cdot \nabla (\partial_{ij} \tilde{u}) , \partial_{ij} \tilde{u} \rangle \right| + \left| \langle | \nabla (\partial_{ij} \tilde{u}) |^2 , \tilde{\theta} \rangle \right| \notag\\
&\qquad\qquad + \left| \langle \partial_{ij} \tilde{u} \Delta \tilde{\theta} , \partial_{ij} \tilde{u} \rangle \right| + \left| \langle \partial_j \tilde{u} \Delta (\partial_i \tilde{\theta}) , \partial_{ij} \tilde{u} \rangle \right| + \left| \langle \partial_i \tilde{u} \Delta (\partial_j \tilde{\theta}) , \partial_{ij} \tilde{u} \rangle \right| \notag\\
&\qquad\qquad + \left| \langle \nabla \tilde{u} \cdot \nabla (\partial_{ij} \tilde{\theta}) , \partial_{ij} \tilde{u} \rangle \right| + \left| \langle \nabla (\partial_{ij} \tilde{\theta}) \cdot \nabla (\partial_{ij} \tilde{u}) , \tilde{u} \rangle \right| \notag\\
&\leq 2 \| \nabla^2 \tilde{\theta} \|_{L^2} \| \nabla^2 \tilde{u} \|_{L^4}^2 + 3 \| \nabla \tilde{\theta} \|_{L^4} \| \nabla^2 \tilde{u} \|_{L^4} \| \nabla^3 \tilde{u} \|_{L^2} \notag\\
&\qquad\qquad + 3 \| \nabla \tilde{u} \|_{L^4} \| \nabla^2 \tilde{u} \|_{L^4} \| \nabla^3 \tilde{\theta} \|_{L^2}  \notag\\
&\qquad\qquad + \| \tilde{\theta} \|_{L^\infty} \| \nabla^3 \tilde{u} \|_{L^2}^2 + \| \tilde{u} \|_{L^\infty} \| \nabla^3 \tilde{\theta} \|_{L^2} \| \nabla^3 \tilde{u} \|_{L^2} \notag\\
&\leq C \| \tilde{\theta} \|_{H^2} \| \nabla^2 \tilde{u} \|_{H^1}^2 + C \| \tilde{\theta} \|_{H^2} \| \nabla^2 \tilde{u} \|_{H^1} \| \nabla^3 \tilde{u} \|_{L^2} \notag\\
&\qquad\qquad + C \| \tilde{u} \|_{H^2} \| \nabla^2 \tilde{u} \|_{H^1} \| \nabla^3 \tilde{\theta} \|_{L^2}  \notag\\
&\quad\quad + \| \tilde{\theta} \|_{L^\infty} \| \nabla^3 \tilde{u} \|_{L^2}^2 + \| \tilde{u} \|_{L^\infty} \| \nabla^3 \tilde{\theta} \|_{L^2} \| \nabla^3 \tilde{u} \|_{L^2} \notag\\
&\leq C\delta \| \nabla^2 (\tilde{u} , \tilde{\theta}) \|_{H^1}.
\end{align}
Similarly, it holds
\begin{align}\label{id_1_22}
\left| \langle \partial_i ( \tilde{\theta} \Delta v + v \Delta \tilde{\theta} ) , \partial_{ij} v \rangle \right| \leq C\delta \| \nabla^2 (v , \tilde{\theta}) \|_{H^1},
\end{align}
and
\begin{align}\label{id_1_23}
\left| \left\langle \partial_{ij} \left( 2 \tilde{\theta} \Delta \tilde{\theta} - \dfrac{\tilde{u}}{2 + \tilde{u}} \Delta \tilde{\theta} + \dfrac{2 \tilde{\theta}^2 + 4 \tilde{\theta} - \tilde{u}}{2 + \tilde{u}} \Delta \tilde{u} \right) , \partial_{ij} \tilde{\theta} \right\rangle \right| \leq C\delta \| \nabla^2 (\tilde{u} , \tilde{\theta}) \|_{H^1}.
\end{align}
For those terms involving $\phi$, we will make use of H\"{o}lder's inequality, Sobolev inequality, Young's inequality and the Hardy-Littlewood-Sobolev inequality (\ref{HLS_ineq}). It follows that
\begin{align}\label{id_1_24}
&\quad\ \left| \left\langle \partial_{ij} \left( 2 \nabla \tilde{\theta} \cdot \nabla \tilde{u} - \nabla v \cdot \nabla \phi - v^2 \right) , \partial_{ij} \tilde{u} \right\rangle \right| \notag\\
&\leq 2 \left| \left\langle \nabla (\partial_{ij} \tilde{\theta}) \cdot \nabla \tilde{u} , \partial_{ij} \tilde{u} \right\rangle \right| + 2 \left| \left\langle \nabla (\partial_i \tilde{\theta}) \cdot \nabla (\partial_j \tilde{u}) , \partial_{ij} \tilde{u} \right\rangle \right| \notag\\
&\qquad\qquad + 2 \left| \left\langle \nabla (\partial_j \tilde{\theta}) \cdot \nabla (\partial_i \tilde{u}) , \partial_{ij} \tilde{u} \right\rangle \right| + 2 \left| \left\langle \nabla \tilde{\theta} \cdot \nabla (\partial_{ij} \tilde{u}) , \partial_{ij} \tilde{u} \right\rangle \right| \notag\\
&\qquad\qquad + \left| \left\langle \nabla (\partial_{ij} v) \cdot \nabla \phi , \partial_{ij} \tilde{u} \right\rangle \right| + \left| \left\langle \nabla (\partial_i v) \cdot \nabla (\partial_j \phi) , \partial_{ij} \tilde{u} \right\rangle \right| \notag\\
&\qquad\qquad + \left| \left\langle \nabla (\partial_j v) \cdot \nabla (\partial_i \phi) , \partial_{ij} \tilde{u} \right\rangle \right| + \left| \left\langle \nabla v \cdot \nabla (\partial_{ij} \phi) , \partial_{ij} \tilde{u} \right\rangle \right| \notag\\
&\qquad\qquad + 2 \left| \left\langle v \partial_{ij} v , \partial_{ij} \tilde{u} \right\rangle \right| + 2 \left| \left\langle \partial_{i} v \partial_j v , \partial_{ij} \tilde{u} \right\rangle \right| \notag\\
&\leq 2 \| \nabla \tilde{u} \|_{L^4} \| \nabla^2 \tilde{u} \|_{L^4} \| \nabla^3 \tilde{\theta} \|_{L^2} + 4 \| \nabla^2 \tilde{\theta} \|_{L^2} \| \nabla^2 \tilde{u} \|_{L^4}^2 \notag\\
&\qquad\qquad + 2 \| \nabla \tilde{\theta} \|_{L^4} \| \nabla^2 \tilde{u} \|_{L^4} \| \nabla^3 \tilde{u} \|_{L^2} + \| \nabla \phi \|_{L^6} \| \nabla^2 \tilde{u} \|_{L^3} \| \nabla^3 v \|_{L^2} \notag\\
&\qquad\qquad + 2 \| \nabla^2 v \|_{L^2} \| \nabla^2 \phi \|_{L^6} \| \nabla^2 \tilde{u} \|_{L^3} + \| \nabla v \|_{L^3} \| \nabla^3 \phi \|_{L^6} \| \nabla^2 \tilde{u} \|_{L^2} \notag\\
&\qquad\qquad + 2 \| v \|_{L^\infty} \| \nabla^2 v \|_{L^2} \| \nabla^2 \tilde{u} \|_{L^2} + 2 \| \nabla v \|_{L^4}^2 \| \nabla^2 \tilde{u} \|_{L^2} \notag\\
&\leq C \| \tilde{u} \|_{H^2} \| \nabla^2 \tilde{u} \|_{H^1} \| \nabla^3 \tilde{\theta} \|_{L^2} + C \| \tilde{\theta} \|_{H^2} \| \nabla^2 \tilde{u} \|_{H^1}^2 \notag\\
&\qquad\qquad + C \| \tilde{\theta} \|_{H^2} \| \nabla^2 \tilde{u} \|_{H^1} \| \nabla^3 \tilde{u} \|_{L^2} + C \| v \|_{L^2} \| \nabla^2 \tilde{u} \|_{H^1} \| \nabla^3 v \|_{L^2} \notag\\
&\qquad\qquad + C \| \nabla^2 v \|_{L^2} \| \nabla v \|_{L^2} \| \nabla^2 \tilde{u} \|_{H^1} + C \| \nabla v \|_{H^1} \| \nabla^2 v \|_{L^2} \| \nabla^2 \tilde{u} \|_{L^2} \notag\\
&\qquad\qquad + 2 \| v \|_{L^\infty} \| \nabla^2 v \|_{L^2} \| \nabla^2 \tilde{u} \|_{L^2} + C \| \nabla v \|_{H^1}^2 \| \nabla^2 \tilde{u} \|_{L^2} \notag\\
&\leq C\delta \left( \| \nabla^2 \tilde{u} \|_{H^1}^2 + \| \nabla v \|_{H^2}^2 + \| \nabla^3 \tilde{\theta} \|_{L^2}^2 \right).
\end{align}
Similarly, we have
\begin{align}\label{id_1_25}
&\quad\ \left| \left\langle \partial_{ij} \left( 2 \nabla \tilde{\theta} \cdot \nabla v - \nabla \tilde{u} \cdot \nabla \phi - \tilde{u} v \right) , \partial_{ij} v \right\rangle \right| \notag\\
&\leq C\delta \left( \| \nabla^2 \tilde{u} \|_{H^1}^2 + \| \nabla v \|_{H^2}^2 + \| \nabla^3 \tilde{\theta} \|_{L^2}^2 \right),
\end{align}
\begin{align}\label{id_1_26}
&\quad\ \left| \left\langle \partial_{ij} \left( \left(c + 1\right) |\nabla \tilde{\theta}|^2 + \left(c + 3\right) \dfrac{1 + \tilde{\theta}}{2 + \tilde{u}} \nabla \tilde{u} \cdot \nabla \tilde{\theta} \right) , \partial_{ij} \tilde{\theta} \right\rangle \right| \notag\\
&\leq C\delta \left( \| \nabla \tilde{u} \|_{H^2}^2 + \| \nabla^2 \tilde{\theta} \|_{H^1}^2 \right),
\end{align}
and
\begin{align}\label{id_1_27}
&\quad\ \left| \left\langle \partial_{ij} \left( 2\ \dfrac{1 + \tilde{\theta}}{2 + \tilde{u}} \nabla v \cdot \nabla \phi + \left(c + 2\right) \dfrac{v}{2 + \tilde{u}} \nabla \tilde{\theta} \cdot \nabla \phi + \dfrac{(1 + \tilde{\theta}) v^2}{2 + \tilde{u}} \right) , \partial_{ij} \tilde{\theta} \right\rangle \right| \notag\\
&\leq C\delta \left( \| \nabla \tilde{u} \|_{H^2}^2 + \| v \|_{H^2}^2 + \| \nabla^2 \tilde{\theta} \|_{H^1}^2 \right).
\end{align}
For the remaining term, we further do integration by parts once and use an interpolation inequality to deduce that
\begin{align}\label{id_1_28}
&\quad\ \left| \left\langle \partial_{ij} \left( |\nabla \phi|^2 \right) , \partial_{ij} \tilde{\theta} \right\rangle \right| \notag\\
&\leq 2 \left| \left\langle \nabla \phi \cdot \nabla ( \partial_{ij} \phi ) , \partial_{ij} \tilde{\theta} \right\rangle \right| + 2 \left| \left\langle \nabla ( \partial_i \phi ) \cdot \nabla ( \partial_j \phi ) , \partial_{ij} \tilde{\theta} \right\rangle \right| \notag\\
&\leq 4 \left| \left\langle \nabla \phi \cdot \nabla ( \partial_{ij} \phi ) , \partial_{ij} \tilde{\theta} \right\rangle \right| + 2 \left| \left\langle \nabla \phi \cdot \nabla ( \partial_j \phi ) , \partial_i \partial_{ij} \tilde{\theta} \right\rangle \right| \notag\\
&\leq 4 \| \nabla \phi \|_{L^2} \| \nabla^3 \phi \|_{L^6} \| \nabla^2 \tilde{\theta} \|_{L^3} + 2 \| \nabla \phi \|_{L^3} \| \nabla^2 \phi \|_{L^6} \| \nabla^3 \tilde{\theta} \|_{L^2} \notag\\
&\leq 4 \| \nabla \phi \|_{L^2} \| \nabla^3 \phi \|_{L^6} \| \nabla^2 \tilde{\theta} \|_{L^3} + 2 \| \nabla \phi \|_{L^2}^{1/2} \| \nabla \phi \|_{L^6}^{1/2} \| \nabla^2 \phi \|_{L^6} \| \nabla^3 \tilde{\theta} \|_{L^2} \notag\\
&\leq C \| \nabla \phi \|_{L^2} \| \nabla^2 v \|_{L^2} \| \nabla^2 \tilde{\theta} \|_{H^1} + C \| \nabla \phi \|_{L^2}^{1/2} \| v \|_{L^2}^{1/2} \| \nabla v \|_{L^2} \| \nabla^3 \tilde{\theta} \|_{L^2} \notag\\
&\leq C \left( \| v \|_{L^2}^2 + \| \nabla^2 \tilde{\theta} \|_{H^1}^2 + \| \nabla \phi \|_{L^2}^2 \right).
\end{align}
Combining (\ref{id_1_19})--(\ref{id_1_28}) and summing over all $i, j = 1, 2, 3$, we obtain
\begin{align}\label{id_1_29}
&\quad\ \dfrac{1}{2} \dfrac{d}{dt} \left( \| \nabla^2 (\tilde{u} , v) \|_{L^2}^2 + 2c \| \nabla^2 \tilde{\theta} \|_{L^2}^2 \right) + C \| \nabla^3 (\tilde{u} , v, \tilde{\theta}) \|_{L^2}^2 + C \| \nabla^2 v \|_{L^2}^2 \notag\\
&\leq C\delta \left( \| \nabla^2 (\tilde{u} , \tilde{\theta}) \|_{L^2}^2 + \| \nabla (\tilde{u} , v) \|_{L^2}^2 + \| v \|_{L^2}^2 + \| \nabla \phi \|_{L^2}^2 \right).
\end{align}
Therefore, Lemma \ref{lem_H1_est} follows (\ref{id_1_18}) and (\ref{id_1_29}).
\end{proof}

Finally, we give the estimate on $\| \nabla \phi \|_{L^2}$.
\begin{lem}\label{nabla_phi_est}
Assuming (\ref{H2_small}) for $\delta$ sufficiently small, we have
\begin{align*}
\dfrac{1}{2} \dfrac{d}{dt} \| \nabla \phi \|_{L^2}^2 + \| v \|_{L^2}^2 + C \| \nabla \phi \|_{L^2}^2 \leq C \delta \| \nabla (v , \tilde{\theta}) \|_{L^2}^2.
\end{align*}
\end{lem}

\begin{proof}
Multiplying (\ref{eqn_1_2}) by $-\phi$, integrating over $\mathbb{R}^3$, and doing integration by parts, we obtain
\begin{align}\label{id_1_30}
&\quad\ \frac{1}{2} \frac{d}{dt} \| \nabla \phi \|_{L^2}^2 \notag\\
&= - \langle \Delta v - 2v , \phi \rangle - \langle \tilde{\theta} \Delta v + v \Delta \tilde{\theta} + 2 \nabla \tilde{\theta} \cdot \nabla v , \phi \rangle + \langle \nabla \tilde{u} \cdot \nabla \phi + \tilde{u} v , \phi \rangle.
\end{align}
By using the Poisson equation for $\phi$,
\begin{align}\label{id_1_31}
\langle \Delta v - 2v , \phi \rangle &= \langle \Delta (\Delta \phi) - 2\Delta \phi , \phi \rangle \notag\\
&= \| \Delta \phi \|_{L^2}^2 + 2 \| \nabla \phi \|_{L^2}^2 = \| v \|_{L^2}^2 + 2 \| \nabla \phi \|_{L^2}^2.
\end{align}
Next,
\begin{align}\label{id_1_32}
\left| \langle \tilde{\theta} \Delta v + v \Delta \tilde{\theta} + 2 \nabla \tilde{\theta} \cdot \nabla v , \phi \rangle \right| &= \left| \langle \Delta (\tilde{\theta} v) , \phi \rangle \right| \notag\\
&\leq \left| \langle \nabla \tilde{\theta} \cdot \nabla \phi , v \rangle \right| + \left| \langle \nabla v \cdot \nabla \phi , \tilde{\theta} \rangle \right| \notag\\
&\leq \| v \|_{L^\infty} \| \nabla \tilde{\theta} \|_{L^2} \| \nabla \phi \|_{L^2} + \| \tilde{\theta} \|_{L^\infty} \| \nabla v \|_{L^2} \| \nabla \phi \|_{L^2} \notag\\
&\leq C\delta \left( \| \nabla (v ,\tilde{\theta}) \|_{L^2}^2 + \| \nabla \phi \|_{L^2} \right).
\end{align}
Again, by using the Poisson equation for $\phi$, integration by parts yields
\begin{align*}
\langle \nabla \tilde{u} \cdot \nabla \phi + \tilde{u} v , \phi \rangle = \langle \nabla \tilde{u} \cdot \nabla \phi + \tilde{u} \Delta \phi , \phi \rangle = - \langle | \nabla \phi |^2 , \tilde{u} \rangle,
\end{align*}
which gives
\begin{align}\label{id_1_33}
\left| \langle \nabla \tilde{u} \cdot \nabla \phi + \tilde{u} v , \phi \rangle \right| \leq \| \tilde{u} \|_{L^\infty} \| \nabla \phi \|_{L^2}^2 \leq C\delta \| \nabla \phi \|_{L^2}^2.
\end{align}
Putting (\ref{id_1_31})--(\ref{id_1_33}) into (\ref{id_1_30}), the proof of Lemma \ref{nabla_phi_est} is completed.
\end{proof}

Proposition \ref{prop_H2_est} is proved by combining Lemmas \ref{lem_L2_est}--\ref{nabla_phi_est}. As a consequence, we complete the proof of Theorem \ref{main_theorem}.

\section{Appendix}

In this section, we present some detailed calculation about the variantional principle used in the model derivation.

\subsection{Least Action Principle}\label{ap_lap}

To obtain the conservative forces in \eqref{conservative force}, integrating with respect to $t$, on both sides of \eqref{ma_con} and \eqref{energy_con}, we have the relation between the state variables $(p,n,e)$ and the flow-maps $(J^p,J^n,J^e)$,
\begin{equation*}
\begin{cases}
\displaystyle p(x,t) = p(x,0) - \nabla \cdot J^p(x,t),\\[3mm]
\displaystyle n(x,t) = n(x,0) - \nabla \cdot J^n(x,t),\\[3mm]
\displaystyle e(x,t) = e(x,0) - \nabla \cdot J^e(x,t).
\end{cases}
\end{equation*}
By definition, $\displaystyle\theta = \frac{e-(p-n)\phi/2}{c_p p + c_n n}$, $\displaystyle \phi(x,t) = \int_{\mathbb{R}^3}  \frac{p(y,t)-n(y,t)}{4\pi |x-y|}dy$. So, their variations satisfy,
\begin{equation*}
\begin{cases}
\displaystyle\delta p(x,t) = -\nabla \cdot \delta J^p,\\[3mm]
\displaystyle\delta n(x,t) = -\nabla \cdot \delta J^n,\\[3mm]
\displaystyle\delta e(x,t) = -\nabla \cdot \delta J^e,\\[3mm]
\displaystyle\delta \phi(x,t) = \int_{\mathbb{R}^3} \frac{\delta p(y,t)-\delta n(y,t)}{4\pi |x-y|}dy,\\[3mm]
\displaystyle \delta \theta(x,t) = \frac{\delta e-(\delta p-\delta n) \phi/2 -(p-n)\delta \phi /2}{c^p p + c^n n} - \theta \frac{c^p \delta p + c^n \delta n}{c^p p + c^n n}.
\end{cases}
\end{equation*}
Unless marked as a function of $(y,t)$, these variables are functions of $(x,t)$. Plugging into \eqref{entropy}, the variation of entropy is,
\begin{align*}
\delta S  =& -\int_{\mathbb{R}^3} \left[[\log p - c^p \log  \theta  + 1] \delta p + [\log n - c^n \log  \theta  + 1] \delta n - \frac{c_p p + c_n n}{ \theta } \delta  \theta  \right]dx,\\
=& \int_{\mathbb{R}^3} \left[\nabla[\log p - c^p \log  \theta  +\frac{\phi}{2\theta}+ \int_{\mathbb{R}^3} \frac{p(y,t)-n(y,t)}{8\pi \theta(y) |x-y|}dy] \cdot \delta J^p(x,t) \right. \nonumber \\ 
 &\quad \left. + \nabla [\log n - c^n \log  \theta  -\frac{\phi}{2\theta}- \int_{\mathbb{R}^3} \frac{p(y,t)-n(y,t)}{8\pi \theta(y) |x-y|}dy] \cdot \delta J^n(x,t) - \nabla \frac{1}{ \theta } \cdot \delta J^e \right] dx.
\end{align*}
So, the variation of the total action,
\begin{align*}
&\quad\ \delta \mathcal{A}[J^p,J^n,J^e] \\
&= \int_0^\tau \int_{\mathbb{R}^3} \left[\nabla \left(\log p - c^p \log  \theta  +\frac{\phi}{2\theta}+ \int_{\mathbb{R}^3} \frac{p(y,t)-n(y,t)}{8\pi \theta(y) |x-y|}dy\right) \cdot \delta J^p(x,t) \right. \nonumber \\ 
&\qquad\qquad\qquad + \nabla \left(\log n - c^n \log  \theta  -\frac{\phi}{2\theta}- \int_{\mathbb{R}^3} \frac{p(y,t)-n(y,t)}{8\pi \theta(y) |x-y|}dy\right) \cdot \delta J^n(x,t) \\
&\qquad\qquad\qquad\qquad\qquad\qquad\qquad\qquad\qquad\qquad\qquad\qquad\qquad\quad \left. - \nabla \frac{1}{ \theta } \cdot \delta J^e(x,t) \right] dxdt.
\end{align*}
This leads to the conservative forces in \eqref{conservative force}. Altough we used the specific form of the pairwise Coulomb potential, the same approach can also be applied to other pariwise interactions.
\begin{rmk}
Alternatively we can define $Q_t = q$ instead of $J^e_t = j^e$ to describe the energy flow-map. However, this would result in tedious computation in the Least Action Principle. The difference between these two approaches, using $Q$ or $J^e$, is in fact a substitution of variables.
\end{rmk}

\subsection{Maxium Dissipation Principle}\label{ap_mdp}

To obtain the dissipative forces in \eqref{dissipative force}, using the definition of \eqref{energy_con}, we have,
\begin{eqnarray*}
\displaystyle q(x,t) &=&   j^e(x,t) - [(c^p+1) \theta +\phi] j^p(x,t) +[(c^n+1) \theta-\phi] j^n(x,t) \nonumber\\
&&\quad +  [\phi(x,t) \nabla \phi_t(x,t)-\phi_t(x,t)\nabla\phi(x,t)]/2,\\
&=&j^e(x,t) - [(c^p+1) \theta +\phi] j^p(x,t) +[(c^n+1) \theta-\phi] j^n(x,t) \nonumber\\
&&\quad -  \frac{1}{2}\int_{\mathbb{R}^3} \left(\frac{\phi(x,t) (x-y)}{4\pi|x-y|^{3}} - \frac{\nabla_x \phi(x,t)}{4\pi |x-y|} \right) \nabla_y \cdot [j^p(y,t) - j^n(y,t)]    dy.
\end{eqnarray*}
So, the variation of the entropy production in \eqref{entropy_production} as a functional of the fluxes,
\begin{align*}
\delta \triangle &= \int_{\mathbb{R}^3} \left[\frac{2j^p \cdot \delta j^p}{ D^p p \theta}+\frac{2j^n\cdot \delta j^n}{ D^n n \theta}+\frac{2q\cdot \delta q}{k \theta^2}\right]dx,\\
&= \int_{\mathbb{R}^3} \left[\left(\frac{2j^p }{ D^p p \theta}-[(c^p+1)\theta+\phi]\frac{2q}{k \theta^2} \right.\right.\\
&\qquad\qquad\qquad \left.\left. + \nabla \int_{\mathbb{R}^3} \frac{q(y,t)}{k \theta^2(y,t)} \left(\frac{\phi(y,t) (y-x)}{4\pi|x-y|^{3}} - \frac{\nabla_y \phi(y,t)}{4\pi |x-y|} \right) dy\right)\cdot \delta j^p \right. \nonumber\\
&\qquad\qquad + \left(\frac{2j^n }{ D^n n \theta}-[(c^n+1)\theta-\phi]\frac{2q}{k \theta^2} \right.\\
&\qquad\qquad\qquad \left. - \nabla \int_{\mathbb{R}^3} \frac{q(y,t)}{k \theta^2(y,t)} \left(\frac{\phi(y,t) (y-x)}{4\pi|x-y|^{3}} - \frac{\nabla_y \phi(y,t)}{4\pi |x-y|} \right) dy\right)\cdot \delta j^n \\
&\qquad\qquad\qquad\qquad\qquad\qquad\qquad\qquad\qquad\qquad\qquad\qquad\qquad \left. +\frac{2q}{k \theta^2} \cdot \delta j^e\right]dx.
\end{align*}

\subsection{Energy conservation in PNPF model}\label{ap_ec}
To verify that \eqref{PNPF} satisfies the First Law of Thermodynamics,
\begin{eqnarray*}
&&\frac{\partial}{\partial t} \left[(c^p p+c^n n)\theta + \frac{(p-n)\phi}{2}\right] \\
&=& (c^p p_t+c^n n_t)\theta + (c^p p+c^n n)\theta_t + \frac{(p_t-n_t)\phi}{2} + \frac{(p-n)\phi_t}{2}, \\
&=& -[c^p \nabla \cdot (pv^p)+c^n \nabla \cdot (nv^n)]\theta + (c^p p+c^n n)\theta_t - \frac{ \phi \nabla^2\phi_t}{2} - \frac{\phi_t\nabla^2\phi}{2}, \\
&=& -\nabla \cdot (c^p  p\theta v^p+c^n n\theta v^n) + \nabla \cdot k\nabla \theta +\frac{p|v^p|^2}{D^p} +\frac{n|v^n|^2}{D^n} - p\theta \nabla \cdot v^p - n\theta \nabla \cdot v^n \\
&&\quad -\phi \nabla^2 \phi_t -\frac{1}{2}\nabla \cdot (\phi_t \nabla \phi - \phi \nabla \phi_t), \\
&=&  -\nabla \cdot (c^p  p\theta v^p+c^n n\theta v^n) - \nabla \cdot (p\theta v^p + n\theta v^n)  - \nabla \cdot (p\phi v^p - n\phi v^n) \\
&&\quad -\frac{1}{2}\nabla \cdot (\phi_t \nabla \phi - \phi \nabla \phi_t) -\nabla \cdot q.
 \end{eqnarray*}
So the total energy $\displaystyle E(t)=\int_{\mathbb{R}^3} e(x,t) dx$ is conserved:
\begin{align*}
\displaystyle \frac{d}{dt} E(t) =0.
\end{align*}

\section{Concluding Remarks}

In this paper, we extend the classical variational principles to describe both the charge transport and temperature evolution, and obtain a Poisson--Nernst--Planck--Fourier system. This approach guarantees the thermodynamic consistency of the model, and can be applied to a variety of non-isothermal complex fluid systems. We then use the energy method to prove the global well-posedness of the PNPF system, under the assumption of the initial data to be sufficiently close to the equilibrium state. 

Future work will include the different form of Clausius-Duhem inequalities, which could result in different set of constitutive relations. On the other hand, the global well-posedness with arbitrary initial condition is challenging, because the PNPF system is not strictly parabolic and there is no direct maximum principle can be applied to obtain the lower bound for temperature, which requires further analysis.

\section{Acknowledgement}

C.-Y. Hsieh and P. Liu would like to thank the Department of Applied Mathematics of the Illinois Institute of Technology for hosting their visits and providing the great working environment.

\end{document}